\documentclass[12pt]{amsart}
\usepackage[english]{babel}
\usepackage{amsfonts,amssymb,latexsym,amscd}

\theoremstyle{plain}
\newtheorem{theorem}{Theorem}

\newtheorem{proposition}{Proposition}
\newtheorem{corollary}{Corollary}
\theoremstyle{definition}
\newtheorem{definition}{Definition}
\theoremstyle{remark}

\newtheorem{example}{Example}

\usepackage{hyperref}

\begin{document}

\title{Groups of Invertible Binary Operations of a Topological Space}
\author{Pavel S. Gevorgyan}

\address{Moscow Pedagogical State University}

\email{pgev@yandex.ru}

\begin{abstract}
In this paper, continuous binary operations of a topological space are studied and a criterion of their invertibility is proved. The classification problem of groups of invertible continuous binary operations of locally compact and locally connected spaces is solved. A theorem on the binary distributive representation of a topological group is also proved.
\end{abstract}


\maketitle

\section{Notation and auxiliary results}

Throughout this paper, by a space we mean a topological space. All spaces are assumed to be Hausdorff. 

By $C(X,Y)$ we denote the space of all continuous maps of the space $X$ to space $Y$, endowed with the compact-open topology, that is, the topology generated by the subbase consisting of all sets of the form $W(K, U)=\{f:X\to Y; \ f(K)\subset U\}$, where $K$ is a compact subset of $X$ and $U$ is an open subset of $Y$.  All spaces of maps are considered in the compact-open topology. 

If $G$ is a topological group, then there is a natural group operation on $C(X,G)$: given any continuous maps $f,g\in C(X,G)$, their product $fg\in C(X,G)$ is defined by formula $(fg)(x)=f(x)g(x)$ for all $x\in X$.

\begin{theorem}[\cite{McCoy}]\label{th_McCoy}
If $G$ is a topological group, then so is $C(X,G)$.
\end{theorem}

The group of all homeomorphisms of $X$ is denoted by $H(X)$. Generally, this group is not a topological
group. However, the following theorem holds.

\begin{theorem}[\cite{Arens}]\label{th_arens}
If $X$ is a locally compact and locally connected space, then $H(X)$ is a topological group.
\end{theorem}

The symmetric group on a set $X$ is denoted by $S(X)$. In the case where $X$ is a finite set, this group is denoted by $S_n(X)$ or $S_n$, where $n$ is the number of elements in $X$. The order of the group $S_n(X)$ is equal to $n!$: $|S_n(X)|=n!$ 

A detailed exposition on the above used notions and results, as well as on other definitions, notions and results, used in this paper without reference, can be found in \cite{Br}--\cite{Gev2}.

\section{Continuous binary operations of topological spaces.}

Let $X$ be a topological space. A continuous map $f:X^2\to X$ is called a continuous binary operation on the space $X$. The set of all continuous binary operations on $X$ we denote by $C_2(X)$. A composition of two binary operations $f, \varphi \in C_2(X)$ is defined by formula:
\begin{equation}\label{eq4}
(f\circ\varphi) (t,x)=f(t, \varphi(t,x)),
\end{equation}
where $t,x\in X$.

If $f:X^2\to X$ is a continuous binary operation, then for every$t\in X$ we define a continuous map
$f_t:X\to X$ by formula:
ft(x) = f(t,x). %
\begin{equation}\label{eq5}
f_t(x)=f(t,x).
\end{equation}

Observe that a continuous binary operation $f:X^2\to X$ can be considered as a family of continuous maps $\{f_t\}$:   $f=\{f_t\}$, which continuously depends on the index $t\in X$. In these notation, the composition of two binary operations $f=\{f_t\}$ and $\varphi=\{\varphi_t\}$, defined in \eqref{eq4}, becomes 
$$f\circ \varphi = \{f_t\circ \varphi_t\},$$
explaining the meaning of formula \eqref{eq4}.

\begin{proposition}
The space $C_2(X)$ is a semigroup with identity element $e(t,x)=x$, that is, a monoid with respect to composition of binary operations.
\end{proposition}

The proof follows by checking the semigroup axioms, and so is omitted.

\begin{definition}
A continuous binary operation $f\in C_2(X)$ is said to be \textit{invertible} if there exists a continuous binary operation $f^{-1}\in C_2(X)$ such that
\begin{equation*}
f\circ  f^{-1}=f^{-1}\circ  f=e.
\end{equation*}
In this case, $f$ and $f^{-1}$ are said to be \textit{mutually inverse binary operations}.
\end{definition}

The subset of all invertible elements of the set $C_2(X)$ we denote by $H_2(X)$. Thus, $H_2(X)$ is a group.

\begin{example}
Let $X=\{a, b\}$ be a two-point discrete space. The symmetric group  $S_2(X)$ of permutations of this space is the cyclic group $\mathbb{Z}_2$, and the group of all invertible binary operations on $X=\{a, b\}$ is the group of order 4 with two generators $\varphi_1$ and $\varphi_2$, which are specified as follows:

\begin{center}
\begin{tabular}{|p{30pt}|p{30pt}|p{30pt}|p{30pt}|}
\hline
\raisebox{-2.80ex}[0cm][0cm]{$\varphi_1$:}&
$b$&
$a$&
 $a$ \\
\cline{2-4}&
$a$&
$b$&
$b$ \\
\cline{2-4}&
 &
$a$&
$b$ \\
\hline
\end{tabular} \quad 
\begin{tabular}{|p{30pt}|p{30pt}|p{30pt}|p{30pt}|}
\hline
\raisebox{-2.80ex}[0cm][0cm]{$\varphi_2$:}&
$b$&
$b$&
 $a$ \\
\cline{2-4}&
$a$&
$a$&
$b$ \\
\cline{2-4}&
 &
$a$&
$b$ \\
\hline
\end{tabular}\vspace{2mm}
\end{center}
This, as it is known, is the Klein four-group.

In the case of a three-point set $X$, the order of the group $H_2(X)$ is equal to $(3!)^3=216$ (see Corollary \ref{cor2} below).
\end{example}

We have the following result, the proof of which is not difficult, and so is omitted.

\begin{theorem}\label{th01}
If a continuous binary operation $f=\{f_t\} \in C_2(X)$ is invertible, then the continuous map $f_t:X\to X$ defined by \eqref{eq5} is a homeomorphism for any $t\in X$, and $f^{-1}=\{f^{-1}_t\}$.
\end{theorem}

The converse of Theorem \ref{th01} is true for locally compact and locally connected spaces.

\begin{theorem}\label{th011}
Let $X$ be a locally compact and locally connected space, and let $f=\{f_t\} :X^2\to X$ be a continuous binary operation. If the map $f_t:X\to X$ is a homeomorphism for every $t\in X$, then the binary operation $f=\{f_t\}$ is invertible, and $f^{-1}=\{f^{-1}_t\}$.
\end{theorem}
    
\begin{proof}
Consider the binary operation $f^{-1}$ given by $f^{-1}(t,x)=f_t^{-1}(x)$, and show that it is a continuous
inverse to $f:X^2\to X$.

We first establish the continuity of the map $f^{-1}:X^2\to X$. Let $(t_0,x_0)\in X^2$ be an arbitrary point,
and let $f^{-1}(t_0,x_0)=f_{t_0}^{-1}(x_0)=y_0$. Let $W\subset X$ be an arbitrary open neighborhood of the point $y_0$ such that the closure $\overline{W}$ is compact. Then there exists a compact connected neighborhood $K$ of the point $x_0$ for which
\begin{equation}\label{KW}
f_{t_0}^{-1}(K)\subset W.
\end{equation}
Denote by $K^{\circ}$ the interior of the set $K$, and observe that 
\begin{equation}\label{xoinko}
f_{t_0}(y_0)=x_0\in K^{\circ}.
\end{equation}

It follows from \eqref{KW} that
\begin{equation}\label{dop}
f_{t_0}(W^C\cap \overline{W})\subset K^C,
\end{equation}
where $W^C$ and $K^C$ are the complements of the sets $W$ and $K$, respectively.

Next, since $f:X^2\to X$ is a continuous binary operation, $y_0$ and $W^C\cap \overline{W}$ are compact, and $K^{\circ}$ and $K^C$ are open subsets of the space $X$, it follows from \eqref{xoinko} and \eqref{dop} that there exists an open neighborhood $U$ of the point $t_0$, such that for every $t\in U$
\begin{equation}\label{xoinko1}
f_{t}(y_0)\in K^{\circ}
\end{equation}
and
\begin{equation*}\label{dop1}
f_{t}(W^C\cap \overline{W})\subset K^C.
\end{equation*}
Hence 
\begin{equation*}\label{dop2}
K\subset f_{t}(W\cup \overline{W}^C)
\end{equation*}
for any $t\in U$. Therefore
$$f_{t}^{-1}(K)\subset W\cup \overline{W}^C.$$
Since $f_{t}^{-1}(K)$ is a connected set, and $W$ and $\overline{W}^C$ are disjoint open sets, it follows from the last inclusion that $f_{t}^{-1}(K)$ is contained in one of the sets $W$ and $\overline{W}^C$. However, it is clear that in view of \eqref{xoinko1} we have $f_{t}^{-1}(K)\subset W$. Therefore, for all $t\in U$ 
\begin{equation}\label{ko}
f_{t}^{-1}(K^\circ)\subset W.
\end{equation}
Thus, for an arbitrary open neighborhood $W$ of the point $y_0=f_{t_0}^{-1}(x_0)$, we have found open neighborhoods $U$ of the point $t_0$ and $K^\circ$ of the point $x_0$ for which \eqref{ko} is satisfied. This proves the continuity of the binary operation $f^{-1}=\{f^{-1}_t\}$.

To complete the proof of the theorem it remains to observe that the continuous binary operation $f^{-1}:X^2\to X$ is inverse to $f:X^2\to X$, which can be verified easily. Theorem \ref{th011} is proved.
\end{proof}

Theorems \ref{th01} and \ref{th011} imply the following invertibility criterion of continuous binary operations on locally compact and locally connected spaces.

\begin{theorem}\label{th0111}
Let $X$ be a locally compact and locally connected space. A continuous binary operation $f=\{f_t\} :X^2\to X$ is invertible if and only if the continuous map $f_t:X\to X$ is a homeomorphism for any $t\in X$.
\end{theorem}

\section{Classification Of Groups Of Invertible Binary Operations}

The next proposition shows that the groups of invertible continuous binary operations are natural extensions of the group of homeomorphisms.

\begin{proposition}\label{prop-1}
The group $H(X)$ of all homeomorphisms of a topological space $X$ is isomorphic (algebraically and topologically) to a subgroup of the group $H_2(X)$ of invertible binary operations. 
\end{proposition}

\begin{proof}
To each  $f\in H(X)$ we associate a continuous map $\tilde{f}:X^2\to X$, defined by $\tilde{f}(t, x)=f(x)$, $t,x\in X$. It is clear that $\widetilde{f^{-1}}=\tilde{f}^{-1}$. Hence $\tilde{f}$ is a continuous invertible binary operation, that is, $\tilde{f}\in H_2(X)$. The correspondence $f \to \tilde{f}$ is the desired isomorphism between the group $H(X)$ and a subgroup of $H_2(X)$. Proposition \ref{prop-1} is proved.
\end{proof}

The next theorem contains a solution of the problem of classification of groups of invertible continuous binary operations of locally compact and locally connected spaces by means groups of homeomorphisms.

\begin{theorem}\label{th1}
Let $X$ be a locally compact and locally connected space. Then the group  $H_2(X)$ is isomorphic (algebraically and topologically) to the group $C(X,H(X))$.
\end{theorem}

\begin{proof}
Consider the map $p:C(X,H(X))\to H_2(X)$ defined by 
\begin{equation*}
p(f)(t,x)=f(t)(x),
\end{equation*}
for $f\in C(X,H(X))$ and  $t,x\in X$. Since for every $t\in X$ the map $f(t) : X \to X$ is a homeomorphism, by Theorem \ref{th0111} the binary operation $p(f):X\times X\to X$ is invertible, that is, it belongs to the group $H_2(X)$.

Now we show that $p$ is a monomorphism. To this end, we take $f,g\in C(X,H(X))$ such that $f\neq g$, and observe that there exists a point $t_0\in X$ such that $f(t_0)\neq g(t_0)$. Since $f(t_0), g(t_0) \in H(X)$, it follows that $f(t_0)(x_0)\neq g(t_0)(x_0)$ for some point $x_0\in X$. Thus, $p(f)(t_0,x_0)\neq p(g)(t_0,x_0)$, implying that $p(f)\neq p(g)$.

Next, observe that the map $p$ is also an epimorphism. Indeed, let $\varphi \in H_2(X)$ be any continuous binary operation. Then by Theorem \ref{th0111}, the map $\varphi_t:X\to X$ defined by $\varphi_t(x)=\varphi(t,x)$, $t, x\in X$, is a homomorphism. It is easy to see that the element $f\in  C(X,H(X))$, determined by the equality $f(t)=\varphi_t$ , is the preimage of the binary operation $\varphi$: $p(f)(t,x)=f(t)(x)=\varphi_t(x)=\varphi(t,x)$.

Thus, the map $p^{-1}:H_2(X)\to C(X,H(X))$ defined by 
$$p^{-1}(\varphi)(t)(x)= \varphi(t,x),$$
for $\varphi \in H_2(X)$ and $t,x\in X$, is inverse to $p:C(X,H(X))\to H_2(X)$.

The map $p$ is a homomorphism, that is, $p(f\circ g)=p(f)\circ p(g)$. Indeed, for any $t,x\in X$ we have 
\begin{multline*}
p(f\circ g)(t,x)=(f\circ g)(t)(x)= (f(t)\circ g(t))(x)=f(t)(g(t)(x))=\\
=f(t)(p(g)(t,x))=p(f)(t,p(g)(t,x))=(p(f)\circ p(g))(t,x).
\end{multline*}

Now we prove the continuity of $p$. Let $W(K\times K', U)$ be any element of the subbase of the compact-open topology on $H_2(X)$, where $U\subset X$ is open and $K,K'\subset X$ are compact subsets of the space $X$. We show that the preimage of the set $W(K\times K', U)$ is the set $W(K, W(K', U))$, which is an element of the subbase of the compact-open topology on $C(X,H(X))$. Indeed, for any $\varphi \in W(K\times K', U)$ and $f=p^{-1}(\varphi)\in  C(X,H(X))$ we have
\begin{multline*}
\varphi \in W(K\times K', U) \iff \varphi(t,x)\in U \iff p(f)(t,x)\in U \iff \\
\iff f(t)(x)\in U \iff f\in W(K, W(K', U)),
\end{multline*}
where $t\in K$ and $x\in K'$ are arbitrary elements, and the continuity of $p$ follows.

The continuity of the inverse map $p^{-1}:H_2(X)\to C(X,H(X))$ can be shown similarly. Theorem \ref{th1} is proved.
\end{proof}

The group of invertible continuous binary operations $H_2(X)$ generally is not a topological group. However, the following results hold.

\begin{corollary}
If $X$ is a locally compact and locally connected space, then $H_2(X)$ is a topological group.
\end{corollary}

\begin{proof}
By theorem \ref{th_arens}, $H(X)$ is a topological group. Therefore, by Theorem \ref{th_McCoy}, $C(X,H(X))$ is also a topological group. Now we can use Theorem \ref{th1} to conclude that $H_2(X)$ is a topological group.
\end{proof}

\begin{corollary}\label{cor2}
Let  $|X|=n<\infty$. Then $|H_2(X)|=(n!)^n$.
\end{corollary}

\begin{proof}
For a finite set $X$, we have $H(X)=S_n(X)$, where $S_n(X)$ is the symmetric group of permutations of the set $X$. Taking into account that $|S_n(X)|=n!$, the result immediately follows from Theorem \ref{th1}.
\end{proof}

\section{Binary Distributive Representations Of Topological Groups}

\begin{definition}
A subgroup $D\subset H_2(X)$ is said to be distributive if for all $x, x', x'' \in X$ and for all $g, h\in D$ the following condition is fulfilled
\begin{equation}\label{eq-distr}
g(h(x,x'), h(x,x''))=h(x,g(x', x'')).
\end{equation}
\end{definition}

\begin{theorem}\label{th-7}
A subgroup $D\subset H_2(X)$ is distributive if and only if for any $g=\{g_t\}, h=\{h_{t'}\}\in D$, $t,t'\in X$, the following equality holds:
\begin{equation}\label{distribut-1}
g_t\circ h_{t'}=h_{g_{t}(t')}\circ g_t.
\end{equation}
\end{theorem}

\begin{proof}
Let $D\subset H_2(X)$ be a distributive subgroup. Then, in view of \eqref{eq-distr}, we obtain
\begin{multline*}
(g_t\circ h_{t'})(x)=g_t(h_{t'}(x))=g_t(h(t',x))=g(t, h(t',x))=h(g(t,t'), g(t,x))=\\
=h_{g(t,t')}(g(t,x))=h_{g_t(t')}(g_t(x))=(h_{g_t(t')}\circ g_t)(x)
\end{multline*}
for any $x\in X$, and hence the equality \eqref{distribut-1} is satisfied.

Now assume that \eqref{distribut-1} is satisfied. Then for any $g,h\in G$ and $t,t',x\in X$ we can write %
\begin{multline*}
h(g(t,t'), g(t,x))=h_{g(t,t')}(g(t,x))=h_{g_t(t')}(g_t(x))=(h_{g_t(t')}\circ g_t)(x)=\\
=(g_t\circ h_{t'})(x)=g_t(h_{t'}(x))=g_t(h(t',x))=g(t, h(t',x)),
\end{multline*}
implying that $D$ is a distributive subgroup. Theorem \ref{th-7} is proved.
\end{proof}

The groups of invertible continuous binary operations are sufficiently rich by distributive subgroups. Moreover, every topological group can be considered as a distributive subgroup of a suitable chosen group of invertible continuous binary operations.

\begin{theorem}[on binary distributive representation of a topological group]\label{th-8}
Every topological group is a distributive subgroup of some group of invertible binary operations.
\end{theorem}

\begin{proof}
Let $G$ be a topological group. Consider the group of invertible binary operations $H_2(G)$ of
$G$, and define the map $i:G\to H_2(G)$, which to each element $g\in G$ associates the binary operation
$i_g\in H_2(G)$, defined by 
$$i_g(h_1,h_2)=h_1gh_1^{-1}h_2,$$
where $g,h_1,h_2\in G$.

Since $i_g(e,e)=g$ for any $g\in G$, where $e$ is the identity element of the group $G$, the map $i$ is a monomorphism.

The map $i$ is also a homomorphism. Indeed, we have
\begin{multline*}
i_{gk}(h_1,h_2)=h_1gkh_1^{-1}h_2=h_1gh_1^{-1}h_1kh_1^{-1}h_2=i_g(h_1, h_1kh_1^{-1}h_2)= \\
=i_g(h_1, i_k(h_1,h_2))=[i_g\circ i_k] (h_1,h_2),
\end{multline*}
where $g,k,h_1,h_2 \in G$.

The continuity of the map $i$ follows from the continuity of operations $(g,h)\to gh$ and $g\to g^{-1}$ for all
$g,h \in G$. Thus, $i$ is an isomorphism of the group $G$ on its image.

Observe that $i(G)$ is a distributive subgroup of the group $H_2(G)$. Indeed, for any $g,h,k,k_1,k_2\in G$
we have the following chain of equalities:
\begin{multline*}
i_g(i_h(k,k_1), i_h(k,k_2))=i_g(khk^{-1}k_1, khk^{-1}k_2)=khk^{-1}k_1gk_1^{-1}kh^{-1}k^{-1}khk^{-1}k_2=\\
=khk^{-1}k_1gk_1^{-1}k_2=i_h(k,k_1gk_1^{-1}k_2)=i_h(k,i_g(k_1, k_2)),
\end{multline*}
and the result follows. Theorem \ref{th-8} is proved.
\end{proof}

Note that Theorem \ref{th-8} is a binary topological version of the Cayley’s classical theorem on representation of an arbitrary finite group by unary operations (permutations).

\end{document}